\documentclass[12pt,a4paper]{article}
\usepackage{amssymb}
\usepackage[centertags]{amsmath}
\usepackage{amscd, amsthm}
\usepackage{stmaryrd}
\usepackage{enumitem}

\usepackage{xspace}

\usepackage{hyperref}

\newtheorem{prop}{Proposition}[section]

\newtheorem{thm}[prop]{Theorem}

\newtheorem{corollario}[prop]{Corollary}

\newtheorem{osserva}[prop]{Remark}

\theoremstyle{remark}

\def\be{\begin{equation}}
\def\ee{\end{equation}}

\def\dotminus{\mathbin{\ooalign{\hss\raise1ex\hbox{.}\hss\cr
  \mathsurround=0pt$-$}}}

\begin{document}
\title{The model theory of residue rings of models of Peano Arithmetic: The prime power case}
\author{P. D'Aquino\footnote{Dipartimento di Matematica e Fisica, Universit\`{a} della Campania ``L. Vanvitelli'', viale Lincoln, 5, 81100 Caserta,  Italy,   e-mail: paola.daquino@unicampania.it } \ and  
A. Macintyre\footnote{Partially supported by a Leverhulme Emeritus Fellowship.
\newline School of Mathematics, University of Edinburgh, King's Buidings,  Peter Guthrie Tait Road
Edinburgh EH9 3FD, UK, e-mail: a.macintyre@qmul.ac.uk \newline Keywords: models of Peano Arithmetic, ultraproducts, local rings, valuations and generalizations} 
}
\maketitle

\begin{abstract}
In \cite{MacResField} the second author gave a systematic analysis of definability and decidability for rings $\mathcal M/p\mathcal M$, where $\mathcal M$ is a model of Peano Arithmetic and $p$ is a prime in $\mathcal M$. In the present paper we extend those results to the more difficult case of $\mathcal M/p^k\mathcal M$, where $\mathcal M$ is a model of Peano Arithmetic, $p$ is a prime in $\mathcal M$, and $k>1$. In \cite{MacResField} work of Ax on finite fields was used, here we use in addition work of Ax on ultraproduct of $p$-adics.

\end{abstract}

\section{Introduction} 

\subsection{} Boris Zilber asked the second author the following question:

If $\mathcal M$ is a nonstandard model of full Arithmetic, and $n$ is a nonstandard element of $\mathcal M$ congruent to $1$ modulo all  standard integers $m$, does the ring $\mathcal M/n\mathcal M$ interpret Arithmetic?

The problem arose in work of Zilber on commutator relations in quantum mechanics, see \cite{BorisComm}.  

Although until recently little attention had been given to rings $\mathcal M/n\mathcal M$, one would naturally expect that $\mathcal M/n\mathcal M$ is much simpler than $\mathcal M$, for example being {\it internally finite in $\mathcal M$}, and thus $\mathcal M$ should not be interpretable  in $\mathcal M/n\mathcal M$, so Zilber's Problem should have a negative answer. 

This is in fact how things turn out. From the outset of our work we looked at much more general problems connected to definability in the rings $\mathcal M/n\mathcal M$, to emphasize the vast gulf between the residue rings $\mathcal M/n\mathcal M$ and $\mathcal M$. We chose to work with models $\mathcal M$ of Peano Arithmetic, where one inevitably encounters G\"odelian phenomena, and contrast $\mathcal M$ with the ring $\mathcal M/n\mathcal M$ where one is in  a strongly non-G\"odelian setting, mainly because of Ax's work from \cite{ax4}, already exploited by Macintyre in \cite{MacResField} to give a detailed analysis of the case when $n$ is prime. 

After a lecture of Macintyre in 2015 Tom Scanlon pointed out that mere non-interpretability can easily be obtained via the interpretation of $\mathcal M/n\mathcal M$ as a nonstandard initial segment $[0,n-1]$, using definable pigeon-hole principle imposed by $\mathcal M$ on $\mathcal M/n\mathcal M$. This is certainly the first completely clear solution to Zilber's Problem. But it seems to have no implications for getting close to exact bounds for complexity of definitions in the $\mathcal M/n\mathcal M$.

\subsection{} Our work on the $\mathcal M/n\mathcal M$ has three clearly defined stages, of increasing difficulty. 

\

\noindent {\bf Stage 1.} $n$ prime. Then if $n$ is a standard prime $p$, $\mathcal M/n\mathcal M \cong\mathbb F_p$, while if $n$ is a nonstandard prime then $\mathcal M/n\mathcal M$ is a pseudofinite field (in Ax's sense \cite{ax4}) of characteristic $0$.  See \cite{MacResField}  for this case, and the cited properties. Pseudofinite fields of characteristic $0$ are models of the theory of all finite prime fields, and this theory $PrimeFin$ has a nice set of first-order axioms which we now sketch briefly.

 A field $K$ is a model of   $PrimeFin$ if 
 
 \smallskip

$1)$ $K$ is perfect

$2)$ $K$ has exactly one extension of each finite degree $n$  

$3_d)$ if $\Gamma$ is an absolutely irreducible plane curve of degree $d$ and $|K|>(d-1)^4$ then $\Gamma$  has a point in $K$

$4_d)$ if $|K|\leq (d-1)^4$ then $|K|=p$ for some prime $p\leq (d-1)^4$ 

\noindent 
for each $0<d\in \mathbb N$.

\smallskip

The theory is decidable, and all models are pseudofinite, in the sense weaker than Ax's, namely that each model is elementarily equivalent to an ultraproduct, possibly principal, of finite fields. Finally, each model of the theory is elementarily equivalent to some $\mathcal M/n\mathcal M$, where $\mathcal M\equiv \mathbb Z$.

Noninterpretability of $\mathcal M$ in any model of the theory can be seen via Hrushovski's result that pseudofinite fields are simple see \cite{udi}. 

\

\noindent {\bf Stage 2.} $n=p^k$, $p$ prime, and $k>1$.  This is covered in the present paper, and is indispensable for Stage 3.

\

\noindent {\bf Stage 3.}  $n$ arbitrary. This will be considered in the last paper of the series. We use the factorization theory in $\mathcal M$ of $n$ as {\it an internal finite product} of prime powers, and an internal representation of $\mathcal M/n\mathcal M$ as $\prod_{p|n} \mathcal M/p^{v_p(n)}\mathcal M$ (see \cite{daqCheb}, \cite{BeraDaq}). This gets combined with an {\it internalized Feferman-Vaught Theorem} (see \cite{FV}), and the work of Stage 2, to get analogues of the main results of Stage 1 and 2, and thereby give a thorough analysis in all cases of the definability theory and axiomatizability of the $\mathcal M/n\mathcal M$, with spin-offs on pseudofiniteness. 

It is notable in Stage 2 that the $\mathcal M/p^k\mathcal M$ are Henselian local rings, and models of a natural set of axioms involving TOAGS (see \cite{DDM}), truncated ordered abelian groups. 

We make heavy use of many classical results on Henselian valuation rings, and {\it truncate } them to get results for our $\mathcal M/p^k\mathcal M$. This is not routine. Though we know a huge amount about the logic of Henselian fields we know rather little about the logic of general Henselian local rings (even familiar finite local rings \cite{mints}). 


\section{Algebraic properties of $\mathcal M/p^k\mathcal M$ }

\subsection{$\mathcal M$ standard}
Already for $\mathcal M$ standard there are nontrivial issues of decidability (and definability). 
Let  $n=p^k$, where $p$ is a prime and $k\geq 1$. It is well-known that 
$\mathbb Z/p^k\mathbb Z$ is an henselian local ring  \cite{Nagata}.  The ideals in the ring of $p$-adic integers $\mathbb Z_p$ are generated by powers of $p$. Moreover, $\mathbb Z/p^k\mathbb Z $ is isomorphic in a natural way to $\mathbb Z_p/p^k\mathbb Z_p $. 
For any fixed prime $p$ there is an existential formula in the language of rings  defining  $\mathbb Z_p$ in $\mathbb Q_p$.
 A uniform definition of $\mathbb Z_p$  in $\mathbb Q_p$ without reference to the prime needs an existential-universal formula, see \cite{CDLM}. 
Ax in \cite{ax4} proved that the theory of the class of valued fields $\mathbb Q_p$ as $p$ varies among the primes is decidable.  So,  the theory of the class of all $\mathbb Z/p^k\mathbb Z$ as $p$ varies over all primes and   $k$ over all positive integers, is also decidable. 

\

\noindent 
{\bf Remark 1.} Whether or not $\mathcal M$ is standard, if $n$ is divisible by  more than one prime, $\mathcal M/n\mathcal M$ is not local.
If $n$ is divisible by finitely many primes  then $\mathcal M/n\mathcal M$ is semilocal. In general for $\mathcal M$ nonstandard, $\mathcal M/n\mathcal M$ is {\it nonstandard semilocal} if $n$ is divisible by infinitely many primes (and there are always such $n$ when  $\mathcal M$ is nonstandard).

\

\noindent {\bf Remark 2.} We note that for any $\mathcal M$, and prime $p$, standard or nonstandard, $\mathcal M$ is not Henselian with respect to the $p$-adic valuation $v_p$ on $\mathcal M$. First, suppose $p\not=2$. If  $v_p$ is Henselian then $1+p$ is a square, so $1+p=y^2$ for some $y\in \mathcal M$. Hence, $(y-1)(y+1)=p$, and so either
\begin{enumerate}[label=\alph*)]
\item
$y-1= 1$ and $y+1=p$

\item
$y-1= -1$ and $y+1=-p$
\item
$y-1= p$ and $y+1=1$
\item
$y-1= -p$ and $y+1=-1$

\end{enumerate}

a) implies $p=3$. b) and c) are impossible. d) implies $p=3$. 

So suppose $p=3$. Then, if Hensel's lemma holds, $1+2\cdot 3=7$ is a square. Clearly, $PA$ implies $7$ is not a square.

Finally, assume $p=2$. If Hensel's lemma holds then $1+8\cdot h$ is a square for all $h$, so $17$ is a square, impossible in $PA$.

\subsection{$\mathcal M$  nonstandard}

Each of the rings $\mathbb Z/p^k\mathbb Z$ can be interpreted in  $\mathcal M$  in a uniform way for each prime $p$ and each positive integer $k$, as can each of the natural maps  $\mathbb Z/p^{k+1} \mathbb Z \rightarrow \mathbb Z/p^k\mathbb Z$.  But it not possible to interpret $\mathbb Z_p$ in $\mathcal M$ as the inverse limit of the $\mathbb Z/p^k\mathbb Z$'s with the associated maps. 

$\L$os' theorem implies  that if $D$ is a nonprincipal  ultrafilter on the set of primes then the ultraproduct $\prod_{D} \mathbb Z_p$   is an Henselian valuation ring, whose maximal ideal is $\prod_D\mu_p$ where $\mu_p$ is the maximal ideal of $\mathbb Z_p$. Also the  value group of the ultraproduct of the local domains is the ultraproduct of the value groups of the $\mathbb Z_p$'s, and so an ultrapower of $\mathbb Z$.  Ax's results in \cite{ax4} are needed to prove  that the residue field $\prod_D \mathbb Z_p/ \prod_D\mu_p= \prod_D\mathbb Z_p/\mu_p$ is a pseudofinite field.
\smallskip

We will first analyze the basic algebraic properties of each of the quotients  $\mathcal M/p^k\mathcal M$, and we will show that also for $\mathcal M$ nonstandard $\mathcal M/p^k\mathcal M$ is a  local Henselian ring. We will then appeal to classical results of model theory of Henselian fields (see \cite{ershov}, \cite{basarab}) to understand the theory of the class of all $\mathcal M/p^k\mathcal M$ for $p$  prime and $k$ positive in $\mathcal M$. We will make use also of some constructions and ideas in  \cite{woods}. 

The valuation  $v_p$ induces a  ``valuation'' (which we will denote  by $v$) on the quotient ring $\mathcal M/p^k\mathcal M$. The residue field of $\mathcal M\big/p^{k}\mathcal M$ is either $\mathbb F_p$ if $p$ is standard, or a characteristic $0$ pseudofinite field if $p$ is nonstandard.

\begin{thm}
\label{henselian}
For general prime $p$ and $k>0$, $\mathcal M\big/p^{k}\mathcal M$ is a local Henselian ring, and the unique maximal ideal is principal. 

\end{thm}

\noindent {\it Proof:} The units in $\mathcal M\big/p^{k}\mathcal M$ are  $m+(p^k)$ where  $m$ is prime to $p$. Clearly, the non units form an ideal, and this is the unique maximal ideal of  $\mathcal M\big/p^{k}\mathcal M$ and is generated by $p+(p^k)$.  

Let $f(x)$ be a monic polynomial over $\mathcal M$, and  $\alpha\in \mathcal M$ such that $v_p({f}(\alpha )) > 0 $ and  $v_p(f'(\alpha ))=0 $.  We want to show that  there exists $\beta \in \mathcal M$ such that $f(\beta)=0$ and $v_p(\alpha-\beta)>0$. We use inside $\mathcal M$ the standard approximation procedure. The informal recursion (as it would be done in $\mathbb N$) puts $\beta_0=\alpha$, and $\beta_1=\beta_0+\epsilon_1$, where $\epsilon_1$ should be chosen judiciously from $\mathcal M$. So, $f(\beta_1)=f(\beta_0+\epsilon_1)=f(\beta_0)+f'(\beta_0)\epsilon_1+o (\epsilon_1^2)$. Choose $\epsilon_1=-\frac{f(\alpha)}{f'(\alpha )}$ (an infinitesimal). 
Hence, $v_p(f(\beta_1))=v_p(f(\beta_0+\epsilon_1))\geq  (v_p(f(\beta_0))^2>0$, and $v_p(f'(\beta_1))=v_p(f'(\beta_0)+\epsilon_1f''(\beta_0)+o(\epsilon_1^2))
=0 $.
We iterate this procedure which can be coded in $\mathcal M$, and we get a sequence $\beta_j$'s of elements of $\mathcal M$,  such that $v_p(f(\beta_j))>0$,   $v_p(f'(\beta_j))=0$, $v_p(f(\beta_j))< v_p(f(\beta_{j+1}))$, and $v_p(\beta_0-\beta_j)>0$. If  $v_p(f(\beta_m))\geq k$ for some $m$, then $\beta_m$ is  a solution of $f$ in  $\mathcal M\big/p^{k}\mathcal M$. By the Pigeonhole Principle  this always  happens since otherwise we would have a definable injective map from $\mathcal M$ into an initial segment $[0,p^k)$, clearly a contradiction. 

\medskip

Note that the condition that the unique maximal ideal is principal is elementary. In  general, in a local ring the unique maximal ideal need not be principal.


\section{Truncations}

The principal ideals of $\mathcal M\big/p^{k}\mathcal M$ are generated by $p^j$ for $0<j\leq k$, and are linearly ordered by the divisibility condition with minimum $(0)$ and maximum $(p)$. 

On the ring $\mathcal M\big/p^{k}\mathcal M$ there is a {\it  truncated valuation} $v$ with values in  $[0,k]$, defined by
$$v(m+(p^k))= \left\{ \begin{array}{ll}

v(m) & \mbox{ if } v(m)<k \\
k  & \mbox{ if } v(m)\geq k
\end{array} \right.$$
\noindent  satisfying 
\begin{enumerate}
\item 
$v(x+y)\geq \min (v(x), v(y))$
\item
$v(xy)=\min (k, v(x)+v(y))$
\item 
$v(1)=0$
\item
$v(0)=k$
\end{enumerate}

We can construe $v$ as a map to a ``truncated ordered abelian group",  henceforward called TOAG.  In \cite{DDM} axioms which are true in initial segments of ordered abelian groups are identified. We list them here. Some are obvious, while others need some calculations. In \cite{DDM} it is also shown that models of these axioms can be realized as initial segments of ordered abelian groups. 
 Notice that in \cite{DDM} it is not specified if  the order is discrete. We are mostly interested in discrete orders and some extra axioms will be added later. The language contains two binary operations $+$ and $\dotminus$,  a binary relation $\leq$, and two constants $0$ and $\tau$. 
 
\medskip

\noindent {\bf Axioms:}
\begin{enumerate}
\item
$+$ is commutative 
\item
$x+0=x$ and $x+\tau = \tau$
\item
$x\leq y$ and $x_1\leq y_1$ imply $x+x_1\leq y+y_1$
\item 
$+$ is associative
\item
$x+y=x+z<\tau$ implies  $y=z$  ({\it cancellation rule})
\item
If $x\leq y<\tau$ then there is a unique $z$ with $x+z=y$
\item
there are $z$ such that $z<\tau$ and $x+z=\tau$, and  $\tau \dotminus x=\min \{z: x+z=\tau \}$
\item
$\tau \dotminus (\tau \dotminus x )=x$
\item
If $0\leq x,y<\tau$ and $x+y=\tau$ then $y \dotminus (\tau \dotminus x ) = x \dotminus (\tau \dotminus y )$
\item
If $x+(y+z)=\tau$ and $y+x<\tau$ then $x \dotminus (\tau \dotminus (y +z))= z \dotminus (\tau \dotminus (x +y))$
\item
If $y+z<\tau$, $x+(y+z)=\tau$, $x+y=\tau$, and $z+(y \dotminus(\tau \dotminus x )) <\tau$
\item 
If $y+x=\tau$ and $y+z<\tau$ then $z+(y \dotminus (\tau \dotminus x))<\tau$
\item
If $y+z=y+x=\tau$ and $z+(y \dotminus (\tau \dotminus x))<\tau$ then $x+(y \dotminus (\tau \dotminus z)) = x\dotminus (\tau  \dotminus y) + z))$
\item 
If $y+z=y+x=\tau$  and  $x+(y \dotminus (\tau \dotminus z)) = \tau$ then $z+(y \dotminus (\tau \dotminus x)) = \tau$
\item 
If $y+z=y+x=x+(y \dotminus (\tau \dotminus z)) = \tau$ then $(y \dotminus (\tau \dotminus x)) \dotminus (\tau \dotminus z)= (y \dotminus (\tau \dotminus z)) \dotminus (\tau \dotminus x)$. 
\end{enumerate}

A proof of the following fundamental result is in  \cite{DDM}. 
\begin{thm}
Let $[ 0,\tau ]$ be a TOAG with $+, \dotminus$ and $\leq$. Then there is an ordered abelian group $(\Gamma, \oplus, \leq_{\Gamma})$ with $P$ the semigroup of non-negative elements, and an element $\tau_{\Gamma}$ in $P$ such that  $([ 0,\tau ], +, \leq)$  is isomorphic to   $([ 0,\tau_{\Gamma} ], \oplus, \leq_{\Gamma})$, where the structure on $[ 0,\tau_{\Gamma} ]$ is the one  induced by  $\Gamma$. 
\end{thm}
Our primary interest  in this paper is in  models of Presburger, and extra conditions are needed for characterizing the TOAGs  which are initial segments of models of Presburger. We expand the language for TOAG with an extra constant symbol $1$. The extra axioms  are also sufficient as shown in the following theorem in \cite{DDM}.

\begin{thm}
\label{TOAGPRES}
A TOAG $[0,\tau] $ with least positive element $1$ is a Presburger TOAG if and only if it satisfies the following conditions
\begin{enumerate}
\item
$[0,\tau] $ is discretely ordered and every positive element is a successor;
\item
for every positive integer $n$ and each $x\in [0,\tau] $ there is a $y \in [0,\tau] $  and an integer $m<n$ such that $x=\underbrace{(y+\ldots +y)}_\text{n \mbox{ {\rm   times}}}+\underbrace{(1+\ldots +1)}_\text{m \mbox{ {\rm  times}}}$.
\end{enumerate}
\end{thm}

\bigskip

Let $PresTOAG$ be the set of axioms in the language containing $+$, $\dotminus$, $\leq$,  $0$ and $\tau$ for Presburger   TOAGs together with all the {\it remainders axioms}  as above. 
\begin{thm}
\begin{enumerate}
\item
$PresTOAG$  is model complete.

\item
$PresTOAG$  is not complete. The complete extensions are in natural correspondence with $\hat{{\mathbb Z}}$, where to $(f(p))_{f(p)\in \mathbb Z_p}$ corresponds $$PresTOAG + \tau \equiv f(p) ( \mbox{mod } p^k)$$ for every $k\geq 1$.

\item
The theory of models of $PresTOAG$ has quantifier-elimination in the language augmented by the relations $\equiv_n$  (divisibility by $n$)  for each $n\geq 2$.
\end{enumerate}

\end{thm}

\noindent
{\it Proof: 1. } Let $S_1, S_2$ be models of $PresTOAG$ and assume there is a monomorphism $S_1\hookrightarrow S_2$. Since $\tau$ is in the language, $S_1$ and $S_2$ have the same top element. Let $S_2$ be  an initial segment of a model  $S$ of Presburger.  Let $\lambda_1,  \ldots,\lambda_m\in S_1$. We claim the type of $\langle \lambda_1,  \ldots,\lambda_m\rangle$ in $S_1$ is the same as the type of $\langle \lambda_1,  \ldots,\lambda_m\rangle$ in $S_2$. By the elimination of quantifiers for Presburger Arithmetic the type of $\langle\lambda_1,  \ldots ,\lambda_m\rangle$ over any $S$ is determined by order and congruence conditions in terms of $\tau$ and $1$, and so it is the same in both $S_1$ and $S_2$.

\smallskip

\noindent
{\it 2.} Clearly the conditions got from $f\in \hat{{\mathbb Z}}$ are consistent. The completeness comes easily from Theorem \ref{TOAGPRES} and the elimination of quantifiers for Presburger Arithmetic.

\smallskip

\noindent
{\it 3.}  This follows from elimination of quantifiers for Presburger Arithmetic.

\section{Henselian local rings}
\subsection{}
In Section 3 we considered a valuation $v$ on the Henselian local ring $\mathcal M/p^k\mathcal M$ with values onto the Presburger TOAG $[0,k]$. In addition, the residue field of the local ring is either $\mathbb F_p$ if $p$ is standard, and otherwise a   characteristic $0$ pseudofinite field. The union of these alternatives says exactly that the residue field is a model of the theory $PrimeFin$. Moreover, $v(p)=1$ if $p$ is standard. A natural problem is whether  a Henselian local ring with these properties is a  quotient of the valuation ring of  a Henselian field with the same residue field, and value group a model of Presburger Arithmetic, and in addition $v(p)=1$ if $p$ is standard.

The first order theory of a Henselian valued field  $(K, v)$, in the language of valued fields, is well understood in the particular cases when the value group is a model of Presburger and  the residue  field is either a characteristic  $0$ pseudofinite field  or the residue field is $\mathbb F_p $ and $v(p)=1$ (i.e. $K$ is unramified).

In these cases,  $Th(K, v)$ is completely determined by the theory of the residue field 
and the theory of the value group. 
As a consequence, one gets that $Th(K, v)$ is decidable if and only if the theory of the residue field and the theory of the value group are decidable.


\medskip

\begin{thm}
\label{henselianR}
  For any $\mathcal M\big/p^{k}\mathcal M$  there is a ring $R$ such that

1) $R$ is a Henselian valuation domain of characteristic $0$ and unramified, and the value group $\Gamma$ of $R$ is a $\mathbb Z$-group (i.e. a model of Presburger), with initial segment $[0,k]$.

2) $\mathcal M\big/p^{k}\mathcal M$ is isomorphic to the quotient ring $R\big/ I$ for some principal ideal $I$ of $R$.

3) the residue field of $R$ is naturally isomorphic to the residue field of $\mathcal M\big/p^{k}\mathcal M$. 

\end{thm}

\noindent
{\it Proof: } Using the MacDowell-Specker Theorem (see \cite{kossakschmerl}), choose $\mathcal N$  a proper elementary end extension of $\mathcal M$. Let $d=p^{\delta}$ for some $\delta\in \mathcal N$ and $\delta>\mathcal M$. By Theorem \ref{henselian} the ring $\mathcal N\big/ p^{\delta}\mathcal N$ is a Henselian local ring with respect to the valuation induced by the $p$-adic valuation on $\mathcal N$. The set $\Delta =\{ x\in \mathcal N: v_p(x)>v_p(a) \mbox{ for all } a\in \mathcal M\}$ is a non-principal prime ideal of $\mathcal N$ and contains $ p^{\delta}$.  Hence, $\mathcal N\big/ \Delta$ is a domain, a local ring with value group $\mathcal M$, and residue field  either $\mathbb F_p$ or a characteristic $0$ pseudofinite field. Moreover,  $\mathcal N\big/ \Delta$   is also Henselian since it is  a homomorphic image of $\mathcal N\big/ p^{\delta}\mathcal N$ which is a  Henselian ring. Let $R=\mathcal N\big/ \Delta$ then $R\big/ p^kR\cong \mathcal M\big/ p^k\mathcal M$ (here we use that  $\mathcal N$ is an end extension of $\mathcal M$). \hfill $\Box$

\bigskip

Notice that the maximal ideal of $R$ is principal due to the fact that divisibility is a linear order on valuation domains, and the value group is discrete.  As already noticed  the maximal ideal of $\mathcal M\big/p^{k}\mathcal M$ is principal generated by $p+(p^k)$.

So far we have identified the $\mathcal M\big/p^{k}\mathcal M$  as  Henselian local rings, with two distinct cases (see \cite{ax1}, \cite{ax2}, \cite{ax3}):

\medskip
\noindent 
{\bf Case 1.} $p$ standard, and $\mathcal M\big/p^{k}\mathcal M$  isomorphic to some $S/\alpha S$, where $S\equiv \mathbb Z_p$, and $\alpha \in S$. In particular, $v_p(p)=1$.

\medskip

\noindent 
{\bf Case 2.} $p$ nonstandard, and $\mathcal M\big/p^{k}\mathcal M$  isomorphic to some $S/\alpha S$, with $S$ elementarily equivalent to a ring of power series with exponents in a $\mathbb Z$-group and coefficients from a pseudofinite field of characteristic $0$, with $\alpha \in S$.

In both cases we have the valuation onto a $PresTOAG$, and we can link up to the results of \cite{mints} where we began to analyze  the set of axioms for such rings.  Moreover, a trivial compactness argument shows that any local ring modelling those axioms is elementarily equivalent to some $\mathcal M\big/p^{k}\mathcal M$, where $\mathcal M\models PA$, see below. 

We note one important point. The rings $\mathcal M\big/p^{k}\mathcal M$ have the special property of recursive saturation, see \cite{MacResField}, so we cannot replace elementary equivalence by isomorphism in the preceding paragraph. However, there are standard resplendency arguments which give the converse when the ring $S/\alpha S$ is countable recursively saturated.


\subsection{}

Note that we still have not proved the converse that any $S/\alpha S$ as in Case 1 and 2 above is elementarily equivalent to some $\mathcal M\big/p^{k}\mathcal M$.

Now, we do this and in addition we obtain decidability results. 

\begin{thm}
\label{elementary theory}
Suppose $S$ is as in Cases 1 and Case 2 above, and $\alpha$ is a non-unit  and $ \alpha \not=0$. Then $S/\alpha S$ is elementarily equivalent to an ultraproduct of $\mathbb Z/p^k\mathbb Z$, for $p$ prime and $k>0$.
\end{thm}

\noindent
{\it Proof:} By Ax-Kochen-Ershov,  $S$ is elementarily equivalent to  an ultraproduct of $\mathbb Z_p$'s. So, $S/\alpha S$ is elementarily equivalent to an ultraproduct of $\mathbb Z_p/\alpha_p\mathbb Z_p$, where $\alpha_p$ is a non-zero, nonunit of $\mathbb Z_p$, and each $\mathbb Z_p\big/\alpha_p\mathbb Z_p$  is isomorphic to some $\mathbb Z/p^h\mathbb Z$. So, $S/\alpha S$ is elementarily equivalent to an ultraproduct of various $\mathbb Z/p^h\mathbb Z$, so it is isomorphic to some $\mathcal M\big/p^{k}\mathcal M$ where $\mathcal M$ is an ultrapower of $\mathbb Z$ and $p$ is a prime in $\mathcal M$ and $k\in \mathcal M$. \hfill $\Box$

\begin{corollario}
The elementary theories of the $\mathcal M\big/p^{k}\mathcal M$ are exactly the elementary theories of the $S/\alpha S$ where $S$ is as in Case 1 and 2. 
\end{corollario}

\noindent
{\it Proof:} By Theorem \ref{henselianR} and Theorem \ref{elementary theory}. 
\hfill $\Box$

Ax's decidability results  for the class of all $\mathbb Q_p$ gives (via interpretability) decidability of the class of all $S/\alpha S$, and so of the class of all $\mathcal M\big/p^{k}\mathcal M$.

We summarize what we have proved so far.  

\begin{thm}
\begin{enumerate}
\item 
$\mathcal M\big/p^{k}\mathcal M$ are pseudofinite (or finite).

\item
The theory of the $\mathcal M\big/p^{k}\mathcal M$ is the theory of the $\mathbb Z\big/p^h\mathbb Z$.

\item
The theory of the $\mathcal M\big/p^{k}\mathcal M$ is decidable.

\end{enumerate}
\end{thm}

\noindent
{\it Proof:}  1) and 2) follow from Theorem \ref{elementary theory}. The theory of $\mathbb Z\big/p^h\mathbb Z$ for a fixed prime $p$ and $h>0$ is decidable by \cite{ax4}.

\section{Axioms, definitions and model-completeness}
Let $R=\mathcal M/p^k\mathcal M$. 

\medskip

\label{k=1}
If $k=1$ then $R$ is a model of the theory of finite prime fields (see \cite{MacResField}, \cite{mints}).
\medskip
The theory $PrimeFin$ of finite prime fields (and thus the theory of the $R$'s)  has uniform quantifier elimination in the definitional extension of ring theory got by adding primitives $Sol_n(y_0,\ldots, y_n)$ expressing $\exists t(y_0+\ldots+ y_nt^n=0)$, see \cite{AdlerKiefe}, and \cite{mints} for rings arising from models of $PA$. Models $R_1$ and $R_2$ are elementarily equivalent if and only if they have the same characteristic and agree on all $Sol_n(\ell_0,\ldots, \ell_n)$ for $\ell_0,\ldots, \ell_n\in \mathbb Z$, see \cite{ax4}.

Some $Th(R)$ are model-complete, some not. The model-complete ones are those where $R$ is either $\mathbb F_p$  or  characteristic $0$, and for each $m$ the unique extension of dimension $m$ is got by adjoining an algebraic number, see \cite{AJmodelcomple}.
One gets an abundance of model-complete examples from Jarden's result in \cite{Jarden} that $ \{ \sigma \in Gal(\mathbb Q^{alg}): Fix (\sigma) \mbox{ is pseudofinite} \}$ has measure $1$. An example of a non  model-complete $R$ is one where $\mathbb Q^{alg}\subseteq R$ (see \cite{ax4}).

\bigskip

In the case of  $k>1$ we need to elaborate the discussion in Section 4. There we showed that $\mathcal M/p^k\mathcal M$ is isomorphic to some $S/\alpha S$ where $\alpha\in S$, $S$ is an unramified Henselian domain, with a value group a $\mathbb Z$-group, residue field  a model of the theory of finite prime fields. 
Conversely, any such $S/\alpha S$ is elementarily equivalent to some $\mathcal M/p^k\mathcal M$ for some $\mathcal M$ model of $PA$. 

So we have to investigate axioms, elementary invariants and model-completeness for all $R$ of the form $S/\alpha S$, where $S$ and $\alpha$ are as before. Each is valued in a Presburger TOAG. Moreover, the TOAG is interpretable (in $\mathcal L_{rings}$) in $R$, by taking the underlying set $\Gamma$  of the TOAG,  
to the set of principal ideals of $R$, linearly ordered by reverse inclusion, so $(1)$ is the least element and $(0)$ is the top element (obviously, we get a linear order since $S$ is a valuation domain).  We get a $\oplus$ on $\Gamma$ via $(\alpha) \oplus (\beta)=(\alpha \beta)$, making $\Gamma$ into a TOAG with  $(1)$ as $0$-element, and $(0)$ as $\infty$-element.  The {\it valuation} maps $\alpha$ to $(\alpha)$. In this way we get a Presburger TOAG. This is a crucial axiom about $R$. The $1$ of the Presburger TOAG is $(p)$ (the maximal ideal). In \cite{mints} we show that the elementary type of the TOAG is given by the {\it Presburger type} of the penultimate element of $\Gamma$ (in our case $(k-1)$). Moreover, any Presburger TOAG is model-complete. 

The maximal ideal $\mu$ of $R$ is of course definable as the set of nonunits and the residue field  of $R$ is interpretable. Naturally, we seek ``truncations'' of the Ax-Kochen-Ershov theorem. On the basis of what we have done above we get the following two basic theorems. 


\begin{thm}
The theory of the class of $\mathcal M/p^k\mathcal M$ ($p$ prime and $k\geq1$) is axiomatized by the following conditions 
\begin{enumerate}
\item
Henselian local ring, valued in a Presburger TOAG, with residue field a model of the theory of {\it finite prime} fields;

\item
If  the characteristic of the residue field is a prime $p$ then the valuation of $p$ is the least positive element of the Presburger TOAG.

\end{enumerate}
\end{thm}

\begin{thm}
The elementary theory of an individual $\mathcal M/p^k\mathcal M$ is uniquely determined by the Presburger type of the penultimate element of the TOAG, and by the elementary theory of the residue field. Given a Presburger type and a residue field, there is an $\mathcal M/p^k\mathcal M$ with the corresponding invariants. 
\end{thm}

\begin{osserva}
\begin{enumerate}
\item
In the above we work with the ring language, since the TOAG condition comes from a condition on principal ideals.
\item
Obviously there are decidability results parallel to the above (of the form, for example, if the residue field is decidable, and the Presburger type of the penultimate element of the TOAG is computable then $\mathcal M/p^k\mathcal M$ is decidable). 

\end{enumerate}
\end{osserva}

\subsection{Quantifier elimination}
The analysis of formulas in the $\mathcal M/p^k\mathcal M$, rather than the mere sentences of the preceding subsection, is quite tricky. It involves a Denef-Pas analysis as in \cite{mints} for sentences, but with some complications, and we give only a sketch of the argument.

By the preceding we are dealing (as far as definability is concerned) with rings $S/\alpha S$ with $S$ a Henselian valuation domain with Presburger value group, with residue field $ k$ a model of the theory of finite prime field, and unramified. We note the substantial result in \cite{CDLM} that $S$ is uniformly $\exists \forall$-definable in the language of rings in its field of fractions $K$. 

It is convenient to fix in $S$ an element $t$ with $v(t)=1$, i.e. the least positive element of the value group. We take $t=p$ if  $k$ has characteristic $p$, and $t$ arbitrary otherwise (note then that the type of $t$ has many possibilities). We can assume $K$ is $\aleph_1$-saturated and so  has an angular component map $ac: K \rightarrow k$. This will not be definable in general unless we have some saturation. See \cite{Pas}, and \cite{mints} for the $PA$ context for the details. Note that our $\mathcal M/p^k\mathcal M$ are recursively saturated, and so if $\mathcal M $ is countable we have an $ac$. 

We now work in the 3-sorted formalism $\mathcal L_{Denef-Pas}$, with sorts of $K$, $k$, and $\Gamma$ (field, residue field, and value group with added $\infty$). We have the usual field formalism on each of $K$, $k$, and on $\Gamma$ the usual $+, - , <, \infty$. In addition we have two trans-sort  primitives, $ac$ and $v$, the former  from $K$ to $k$, and the latter from $K$ to $\Gamma \cup \{ \infty \}$. $K$ has characteristic $0$, $K$ has constant $t$ as above with $t=p$ if $k$ has characteristic $p$, and $v(t)=1$.

Our purpose is to analyze the structure of the sets $$\{   (y_1, \ldots ,y_n) \in (S/\alpha S)^n:   S/\alpha S\models \psi (y_1, \ldots ,y_n)  \}$$
for a ring formula $\psi$.

For convenience in applying Denef-Pas, we work in $K$ rather than in $S$, but we exploit the uniform definability of $S$ in $K$. So, we consider variables $x_1, \ldots ,x_n,\alpha$ and the formula expressing (in $K$)  $x_1, \ldots ,x_n,\alpha$ are in $S$, and 
\begin{equation}
\label{Smodels}
S/\alpha S\models \psi (x_1+\alpha S, \ldots ,x_n+\alpha S). 
\end{equation}

   This is a first order condition. Thus by Denef-Pas it is uniformly, except for finitely many $p$ (characteristic of $k$), equivalent to a formula $\Theta (x_1, \ldots ,x_n,\alpha)$ in $\mathcal L_{DP}$ with no bounded $K$ variables (but do not forget the constant $t$!). The basic formulas out of which $\Theta$ is built are:

algebraic equations in $x_1+\alpha S, \ldots ,x_n+\alpha  S, t$ over $K$;

residue field formulas in $ac(\tau_{\ell}) $ for various polynomials $\tau_{\ell}$ over $\mathbb Z$ in $x_1+\alpha S, \ldots ,x_n+\alpha S, t$.

Presburger formulas in $v(\mu_m)$ for various polynomials $\mu_m$ in $x_1+\alpha S, \ldots ,x_n+\alpha S, t$.

\smallskip

In fact, $\Theta$ can be taken as a Boolean combination of these basic formulas, as can be seen by inspection of the proof of Denef-Pas. 

For the finitely many exceptional standard primes $p_1,\ldots ,p_n$ (which can be found effectively from $\psi$, by \cite{ax4}), $S$ is elementarily equivalent to one of $\mathbb Z_{p_1}, \ldots , \mathbb Z_{p_n}$ (the analogous easier argument for sentences is sketched in \cite{mints}). By using Macintyre's quantifier elimination for each of $\mathbb Z_{p_1}, \ldots , \mathbb Z_{p_n}$, and changing the $ac$ if need be, one easily gets the power conditions $P_{\ell}(x)$ to have the required Denef-Pas form. By using that  $p=0$ in $k$ captures the characteristic $p$ condition, one can combine these finitely many analyses with the one that works except for $p_1,\ldots ,p_n$ to get the Denef-Pas result for all $\psi$, giving a new  $\Theta$ that works always, independent of $p$.

Now we can do further elimination semplification in the other sorts. For the Presburger sort we simply have the classical elimination down to order and congruence conditions, provided we have a constant  for the least positive element (and above we have stipulated that $v(t)=1=$ least positive element). So Presburger quantifiers get eliminated. This leaves the issue  of quantification over $k$. Recall that $k$ ranges over models of the theory of finite prime fields. By \cite{ax4} and \cite{AdlerKiefe} one easily sees that (uniformly) $k$ has quantifier elimination down to the solvability predicates $Sol_n$. 

Now recall that we seek elimination results in the $S/\alpha S$, where the above takes place in $K$. However, $k$ depends only on $S$ and not on $\alpha$, and the Presburger conditions have the same value in $S/\alpha S$ as in $S$, for $x_1,\ldots ,x_n$ proper dividing $\alpha$, so in fact we have proved (recall our starting point (\ref{Smodels})):

\begin{thm}
\label{elimination}
Assume the previous conditions on $S$, and adjoin constant $t$ with $v(t)=1$, and $t=p$ if $k$ has characteristic $p$.  Then uniformly in $S$ for any $\psi (x_1,\ldots ,x_n)$ in $\mathcal L_{ring,t}$, the language  of rings with $t$,  there is another such formula $\psi^+  (x_1,\ldots ,x_n,y)$ such that if  $S$ is $\aleph_1$-saturated $S$ has an angular component $ac$ such that for all $\alpha\in S$, $\alpha \not=0$ and $\beta_1, \ldots , \beta_n\in S$  properly dividing $\alpha$ 
$$S/\alpha S \models \psi (\beta_1+\alpha S,\ldots ,\beta_n+\alpha S) \mbox{ \hspace{.1in} } \Leftrightarrow  \mbox{ \hspace{.1in} } S \models \psi^+ (\beta_1,\ldots ,\beta_n,\alpha)
$$
where $\psi^+  (x_1,\ldots ,x_n,y)$ is a Boolean combination of three kinds of sorted formulas (where now the sorts are local ring, residue ring, TOAG):
\begin{enumerate}
\item
polynomial equations over $\mathbb Z[t]$ in $\beta_1+\alpha S,\ldots ,\beta_n+\alpha S$
\item
solvability conditions over $\mathbb Z$ in $\beta_1+\alpha S,\ldots ,\beta_n+\alpha S$ and $ac(\beta_1), \ldots , ac(\beta_n)$, $ac(t)=1$ (i.e. using the predicates $Sol_n$)
\item
Presburger conditions over monomials in $v(\beta_1), \ldots , v(\beta_n)$ in the Presburger TOAG $[ 0, v(t/\alpha)]$. 
\end{enumerate}
\end{thm}

\begin{proof}
This is sketched above, and is the obvious ``truncated" analogue of Denef-Pas.  Note that \cite{DDM} contains the background for 3.
\end{proof}

We do not attempt to go any deeper to minimize the role of the $ac$. This may be worthwhile, but it is not needed for our last topic below. 

\subsection{Elementary embeddings  and model-completeness}
 We work with rings $R\equiv S/\alpha S$, where $S$ is as above. The TOAG valuation is algebraically interpretable in $R$. From Section 4, the elementary type of $R$ is determined by the elementary type of the residue field, and the Presburger type of the penultimate element of the TOAG. For $R_1\equiv R_2$ as rings, and an embedding $f:R_1\rightarrow R_2$, we want to find out when $f$ is elementary. There is no loss of generality in analyzing the case when $f$ is a ring inclusion. It has no chance of being elementary if the least positive element of the TOAG of $R_1$ is not the least positive element of the TOAG of $R_2$. So we work in  $\mathcal L_{ring,t}$, the language of rings with $t$,  with corresponding assumptions on $R_1$ and $R_2$. Thus there is a natural inclusion of residue fields $k_1\rightarrow k_2$, provided each $R_i$ satisfies that $t$ generates the maximal ideal. This we now assume (recall that the corresponding  maximal ideals  $\mu_1$ and $\mu_2$ are both generated by $t$). Note one cost of adjoining $t$ is that our work has to take some account  of part of ``the type of $t$". We indicate, as we go along, what is involved.

Our purpose is to show that if the embedding $k_1\rightarrow k_2$ is elementary then so is $R_1\rightarrow R_2$. Our convention about $t$ ensures that the induced map on TOAGs is elementary. 

Now, neither $R_1$ nor $R_2$ need have an $ac$, as required for the Denef-Pas analysis, and we have to resort to ``tricks of the saturation trade" to reduce to the case when $R_2$ has an $ac$ which restricts to $R_1$. If there is any counterexample to our claim that if $k_1\rightarrow k_2$ is elementary then so is $R_1\rightarrow R_2$, we select such a counterexample (witnessed by a particular residue-field formula), and by a standard compactness/saturation argument produce a counterexample with stronger properties, namely those given in the next two paragraphs.

We are assuming $R_1\subseteq R_2$ (and in fact a local extension because of the $t$-convention), $R_1\cong S/\alpha S$, $R_1 \equiv R_2$ in the rings language, $S$ is $\aleph_1$-saturated, and $k_1\preceq  k_2$ (note that since $ac(t)$ can be chosen as $1$, we need only consider $k_1\preceq  k_2$ in the ring language), and some formula $W(x_1, \ldots ,x_n)$ witnessing that $R_1\not\preceq  R_2$ (in $\mathcal L_{ring,t}$), i.e. there are some $c_1,\ldots , c_n$ in $R_1$ so that $R_1\models W(c_1,\ldots , c_n)$ and $R_2\models \neg W(c_1,\ldots , c_n)$. 

Then we have to work (quite hard) to get an $ac$ on $R_2$ restricting to one on $R_1$. This involves modifying $R_2$ in general. First get $ac$ on $R_1$ using the $\aleph_1$-saturation of $R_1$, recall that $S$ is $\aleph_1$-saturated and by \cite{Cherlin} has a normalized cross-section and thus an $ac$ (appropriately normalized) which truncates to $R_1$. 

Now go from $R_2$ to an $|R_2|^+$-saturated elementary extension $\overline{R_2}$. This  by an obvious adaptation of Cherlin's argument in \cite{Cherlin} give an $\overline{ac}$ on $\overline{R_2}$ extending the $ac$ on $R_1$. Now use $R_1\rightarrow \overline{R_2}$, which still satisfies the original condition that $(R_1,R_2)$ did. By Theorem \ref{elimination} the truncated Denef-Pas version gives us $R_1\preceq \overline{R_2}$ since $k_1\preceq \overline{k_2}$, where  $\overline{k_2}$ is the residue field of $\overline{R_2}$, and all polynomial equations $$f(\eta_1,\dots, \eta_n, ac(\eta_1),\dots, ac(\eta_n))=0,$$ with $\eta_1,\dots, \eta_n\in R_1$ maintain their truth value between $R_1$ and $\overline{R_2}$ (trivially). But since $R_1\preceq \overline{R_2}$ and $R_2\preceq \overline{R_2}$  we must have $R_1\preceq  R_2$.

So we have now the analogue of the result due to Ziegler \cite{Ziegler} in generality.

\begin{thm}
Assume $R_1\subseteq R_2$ with the $t$ condition to guarantee $k_1\subseteq k_2$. Then $R_1\preceq R_2$ if and only if $k_1\preceq k_2$.
\end{thm}

 Finally this gives us a model-completeness result.
 
 \begin{thm}
 The theory of $R$ in the $t$-formalism is model-complete if and only if the theory of the residue field $k$ is.  
 \end{thm}

\section{Neostability and interpretability}

If $p$ is standard, the preceding shows that each $\mathcal M\big/p^{k}\mathcal M$ is interpretable in an ultrapower of $\mathbb Z_p$.  By \cite{delon} the ultrapower is $NIP$, and so does not interpret even $I\Delta_0$ (which has $IP$), a much weaker system than Peano Arithmetic see \cite{daqCheb}.  

When $p$ is nonstandard, $\mathcal M\big/p^{k}\mathcal M$ interprets $\mathcal M\big/p\mathcal M$ which has $IP$ (see \cite{Duret}), so $\mathcal M\big/p^{k}\mathcal M$ has $IP$. However, $\mathcal M\big/p^{k}\mathcal M$ lives in the $NTP_2$ enviroment of neostability \cite{ChernikovKaplanSimon}, since $\mathcal M\big/p^{k}\mathcal M$ is interpretable (by the preceding) in the ring of power series in $\mathcal M\big/p\mathcal M$ with value group a model of Presburger. By \cite{ChernikovKaplanSimon} this ring of power series has $NTP_2$.

However, any model $\mathcal M$ of $I\Delta_0$  has $TP_2$, as we see by the following construction. Let $a_{nm}$ be $p_n^m$ for $n,m$ positive standard integers, and $p_n$ a prime in $\mathcal M$. Consider the formula (of ring theory) $\varphi (x,y)$ saying that $y$ is a power of a prime $p$, and $v_p(x)=v_p(y)$. From \cite{daqCheb} this is given in $\mathcal M$ by a $\Delta_0$-formula. Now, 
\begin{enumerate}
\item 
the set $\{ \varphi (x,a_{nm}): m\in \mathbb N\}$ is inconsistent, for fixed $n$;
\item
for any $f:\omega \rightarrow \omega$ the set $\{ \varphi (x,a_{nf(n)}): n\in \mathbb N\}$ is consistent.
\end{enumerate}

1. This is clear, since the type forces $v_{p_n}(x)=m$, for all $m$.

2. One shows that for each $\overline{n}$, the set $\{ \varphi (x,a_{nf(n)}): n\leq \overline{n}\}$ is realized in $\mathcal M$, in fact by the element $b=p_0^{f(0)}\cdot \ldots \cdot p_{\overline{n}}^{f(\overline{n})}$. 

Since $TP_2$ is preserved by interpretation, we have 
\begin{thm}
For each $k\geq 1$ no model of $I\Delta_0$ is interpretable in any $\mathcal M\big/p^{k}\mathcal M$.
\end{thm}


\medskip

 \noindent
{\bf Concluding remarks.} In the sequel to this paper we consider $\mathcal M\big/n\mathcal M$ for general $n$. The case when $n$ has only finitely many prime divisors is no harder that what we did above. But the general case is much harder.

\end{document}